\documentclass[11pt]{article}
\usepackage{fullpage, amssymb, amsthm, amsmath, bm, verbatim}
\usepackage{hyperref}
\usepackage{cancel,cleveref,makecell}
\usepackage{caption}
\newtheorem{thm}{Theorem}
\newtheorem{thmA}{Theorem}
\newtheorem{lem}[thm]{Lemma}
\newtheorem{lemA}[thmA]{Lemma}

\newtheorem{conj}[thm]{Conjecture}

\newtheorem{ques}[thm]{Question}
\newtheorem*{mainthm}{Main Theorem}
\newtheorem*{keylem}{Key Lemma}
\theoremstyle{definition}

\newtheorem{example}{Example}

\newcommand\vph{\varphi}

\RequirePackage{marginnote,hyperref}
\newcommand\chisl{\chi^{\star}_{\ell}}
\newcommand\chil{\chi_{\ell}}
\newcommand\chisc{\chi^{\star}_{c}}
\newcommand\chic{\chi_{c}}

\renewcommand\ge{\geqslant}

\renewcommand\leq{\leqslant}
\renewcommand\le{\leqslant}

\newcommand\B{\mathcal{B}}
\newcommand\C{\mathcal{C}}
\newcommand\F{\mathcal{F}}
\newcommand\G{\mathcal{G}}
\newcommand\K{\mathcal{K}}

\newcommand\MM{\textbf{M}}

\usepackage{tikz}
\usetikzlibrary{positioning}
\usetikzlibrary{decorations.pathmorphing}

\title{Disjoint Correspondence Colorings\\ for $K_5$-Minor-free Graphs}
\author{
	Wouter Cames van Batenburg\thanks{ D\'epartement d’Informatique, Universit\'e libre de Bruxelles, Belgium; \texttt{w.p.s.camesvanbatenburg@gmail.com}. Supported by the Belgian National Fund
for Scientific Research (FNRS).}
\and
	Daniel W. Cranston\thanks{Department of Mathematics, College of William \& Mary, Williamsburg, VA, USA;
	\texttt{dcransto@gmail.com}.}
\and 
	Franti\v{s}ek Kardo\v{s}
	\thanks{LaBRI, CNRS, University of Bordeaux, Talence, F-33405, France; Department of Computer Science, Faculty of Mathematics, Physics and Informatics, Comenius University, Mlynsk\'{a} Dolina, 842 48 Bratislava, Slovakia; \texttt{frantisek.kardos@u-bordeaux.fr}.}
}

\begin{document}
\maketitle
\begin{abstract}
        Thomassen famously proved that every planar graph is 5-choosable.  We explore variants of this 
	result, focusing on finding disjoint correspondence colorings, in the more general class of 
	$K_5$-minor-free graphs.  Correspondence colorings generalize list colorings as follows.
	Given a graph $G$ and a positive integer $t$, a correspondence $t$-cover $\MM$ assigns 
	to each $v\in V(G)$ a set of allowable colors $\{1_v,\ldots,t_v\}$ and to each edge $vw\in E(G)$
	a matching between $\{1_v,\ldots,t_v\}$ and $\{1_w,\ldots,t_w\}$. An $\MM$-coloring $\vph$
	picks for each vertex $v$ a color $\vph(v)$ (from the set $\{1_v,\ldots,t_v\}$) such that for
	each edge $vw\in E(G)$ the colors $\vph(v),\vph(w)$ are not matched to each other.
	Two $\MM$-colorings $\vph_1,\vph_2$ of $G$ are \emph{disjoint} if $\vph_1(v)\ne\vph_2(v)$
	for all $v\in V(G)$.  For every $K_5$-minor-free graph $G$ and every correspondence 6-cover 
	$\MM$ of $G$, we construct 3 pairwise disjoint $\MM$-colorings $\vph_1,\vph_2,\vph_3$.
	In contrast, we provide examples of $K_5$-minor-free graphs and correspondence 5-covers $\MM$ 
	that do not admit 3 disjoint $\MM$-colorings.
\end{abstract}

\section{Introduction}
\subsection{Definitions and Examples}
Correspondence coloring generalizes list coloring, and was introduced to resolve a longstanding open 
question \cite{DP} on the list-chromatic number of certain planar graphs.  Recall that a list assignment 
$L$ for a graph $G$ gives to each vertex $v$ of $G$ a set $L(v)$ of allowable colors.  An \emph{$L$-coloring} 
is a proper coloring $\vph$ such that $\vph(v)\in L(v)$ for all vertices $v$.  If $|L(v)|=t$ for all $v$, 
then $L$ is called a $t$-assignment and an \emph{$L$-packing}
consists of $L$-colorings $\vph_1,\ldots,\vph_t$ such that $\vph_i(v)\ne \vph_j(v)$ for all vertices $v$
and all distinct $i,j\in[t]$.  That is, the colors $\vph_i(v)$ partition $L(v)$ for each vertex $v$.

	For a graph $G$, a \emph{correspondence $t$-cover} $\MM$ assigns 
	to each $v\in V(G)$ a set of allowable colors $\{1_v,\ldots,t_v\}$ and to each edge $vw\in E(G)$
	a matching between $\{1_v,\ldots,t_v\}$ and $\{1_w,\ldots,t_w\}$. An $\MM$-coloring $\vph$
	picks for each vertex $v$ a color $\vph(v)$ (from the set $\{1_v,\ldots,t_v\}$) such that for
	each edge $vw\in E(G)$ the colors $\vph(v),\vph(w)$ are not matched to each other.
	(It is straightforward to verify that correspondence coloring generalizes list coloring.)
	Two $\MM$-colorings $\vph_1,\vph_2$ of $G$ are \emph{disjoint} if $\vph_1(v)\ne\vph_2(v)$ for all 
	$v\in V(G)$.  An $\MM$-packing for $G$ consists of disjoint $\MM$-colorings $\vph_1,\ldots, \vph_t$.

	The \emph{list-chromatic number}, $\chil(G)$, and \emph{list packing number},
	$\chisl(G)$, are the minimum values of $t$ such that for every $t$-assignment $G$ admits,
	respectively, a list coloring and a list packing.  Analogously, for correspondence $t$-covers, we 
	define the \emph{correspondence chromatic number}, $\chic(G)$, and the \emph{correspondence packing 
	number}, $\chisc(G)$. List packing and correspondence packing were introduced in~\cite{CCvBDK}, where
	the authors discussed connections with other areas of discrete math.  These parameters have since 
	been studied in numerous papers~\cite{BCK25,CCvB,CCvBDK2,CCvBZ,CH23,CDKH24,CSr,KMMP22,KM25,M23}.

	As a warmup with these definitions, we provide a few examples.  Note that always $\chi(G)\le \chil(G)
	\le \chic(G)\le \chisc(G)$; and also $\chil(G)\le \chisl(G)\le \chisc(G)$.  For a graph class $\G$
	and a graph parameter $f$, let $f(\G):=\max_{G\in \G}f(G)$.  Let $\F$ denote the class of forests,
	$\C$ the class of cycles, and $\C_e$ the class of cycles of even length.
	So $2=\chi(\F)=\chil(\F)=\chisl(\F)=\chic(\F)=\chisc(\F)$.  The lower bound is trivial, and the 
	upper bound follows from \Cref{ext-lem} below.  It is well-known (and easy to check) that 
	$\chil(\C_e) = 2 < 3 = \chil(\C)$.  But the reader may find the following more surpising:
	$\chisl(\C_e)=\chisl(\C)=3$, $\chic(\C_e)=\chic(\C)=3$, and $\chisc(\C_e)=\chisc(\C)=4$. 
	In the first case, the upper bound requires a bit of work.  But in the second and third cases, 
	the upper bounds hold because the graphs are $2$-degenerate; in 
	particular, the last of these follows from \Cref{ext-lem} below.  The lower bounds on both
	$\chic(\C_e)$ and $\chisc(\C_e)$ come from constructing a cover that violates a necessary
	parity-type invariant needed to find the desired coloring or packing.
	For more details on these lower bounds, see \cite{CCvBDK2} or the introduction of~\cite{CSr}.

\subsection{Main Result}
The authors of \cite{CCvBDK} proposed the problem of determining $\chisc(\G)$ for various graph classes 
$\G$.  Let $\K_r$ denote the class of all graphs that are $K_r$-minor-free.  It is easy to check that 
$\chisc(\K_2)=1$, $\chisc(\K_3)=2$, and $\chisc(\K_4)=4$; the final upper bound follows from Lemma~\ref{ext-lem}
below because all graphs in $\K_4$ are 2-degenerate.  So the first open case is $\chisc(\K_5)$.
The best known upper bound is $\chisc(\K_5)\le 8$.  This follows from \cite[Theorem~1.7]{CCvBZ} because every
$K_5$-minor-free graph has maximum average degree less than 6.  We cautiously believe that this upper bound 
can be improved as follows.

\begin{conj}\label{conj1}
	If $G$ is $K_5$-minor-free and $\MM$ is a correspondence 6-cover, then $G$ has 6 disjoint 
	$\MM$-colorings.  That is, $\chisc(\K_5)=6$.
\end{conj}

Dvo\v{r}\'{a}k, Norin, and Postle~\cite{DNP19} called a graph \emph{weighted $\epsilon$-flexible} with 
respect to a list-assignment $L$ if there exists a probability distribution on the $L$-colorings $\phi$ 
of $G$ such that $\mathbb{P}(\phi(v)=c) \ge \epsilon$ for every vertex $v$ and color $c\in L(v)$.
In other words, weighted $\epsilon$-flexibility ensures that at each vertex each color has a decent chance 
to appear in the random list-coloring.  Dvo\v{r}\'{a}k, Norin, and Postle  proved that there is an absolute 
constant $\epsilon>0$ such that all planar graphs are weighted $\epsilon$-flexible with respect to every 
$7$-assignment, and they posed as an open problem to show the same for $6$-assignments or even 
$5$-assignments.  While this problem remains wide open, for $7$-assignments the value of $\epsilon$ has been 
improved considerably from $7^{-36}$ to $2^{-7}$ by~\cite{KMMP22} and~\cite{CCvB}.
As detailed in~\cite{CCvBZ} and~\cite{KMMP22}, every graph $G$ is weighted $\frac{1}{\chisl(G)}$-flexible 
with respect to every $\chisl(G)$-assignment. Thus a proof of Conjecture~\ref{conj1}  will make significant 
progress: it will imply that every planar graph (which is $K_5$-minor-free) is weighted 
$\frac{1}{6}$-flexible with respect to every $6$-assignment.
\bigskip

The main goal of this note is to prove the following result, in support of the above conjecture.

\begin{mainthm}
	\label{thm1}
	Let $G$ be a $K_5$-minor-free graph.  (i) For every correspondence 6-cover $\MM$ of $G$, 
	there exist disjoint $\MM$-colorings $\vph_1,\vph_2,\vph_3$.  That is, for all $v\in V(G)$ 
	the colors $\vph_1(v),\vph_2(v),\vph_3(v)$ are pairwise distinct.
	(ii) But the analogous statement is false if we replace 6-cover with 5-cover.
\end{mainthm}

The fact that $K_5$-minor-free graphs are $5$-degenerate easily implies that they have list-chromatic number 
at most $6$, which is  close to the optimal bound $5$. However, for the correspondence packing number this 
approach performs quite poorly.  As discussed above, $\chisc(\C) \le 4$ due to the $2$-degeneracy of cycles.
More generally, by \Cref{ext-lem} we have $\chisc(G)\leq 2d$ for every $d$-degenerate graph $G$, and this is 
sharp \cite[Prop. 24]{CCvBDK} for every $d \ge 1$, due to large complete bipartite graphs. Therefore a vanilla 
degeneracy argument can merely yield $\chisc(\K_5) \le  10$, while it is known that $\chisc(\K_5)\le 8$. This 
indicates that establishing Conjecture~\ref{conj1} is nontrivial. On the lower bound side, note that the 
second part of our Main Theorem implies that $\chisc(\K_5)\ge 6$.
\bigskip

\subsection{Proof overview}
When constructing disjoint list colorings and correspondence colorings, it is often convenient 
to proceed by induction.  Given a graph $G$, a correspondence $t$-cover $\MM$ of $G$
and disjoint partial $\MM$-colorings $\vph^0_1,\ldots,\vph^0_s$, for each vertex $v$
that is not yet colored we seek a color $\vph_j(v)$ to extend $\vph^0_j$ to $v$.
We write $L_j(v)$ for the subset of $[t]$ that is not yet forbidden from use as 
$\vph_j(v)$, due to the colors already used on $N(v)$ and the matchings $\MM$.
The following lemma, from~\cite{CCvBDK}, is generally useful for handling vertices of low degree.
Its proof is instructive, so we include it.

\begin{lemA}[\cite{CCvBDK}]
	\label[lem]{ext-lem}
	Let $G$ be a graph with a correspondence $t$-cover $\MM$ and fix $v\in V(G)$.
	If $d(v)\le t/2$, then any disjoint $\MM$-colorings $\vph_1',\ldots,\vph_s'$
	of $G-v$ extend to disjoint $\MM$-colorings $\vph_1,\ldots,\vph_s$ of $G$.
\end{lemA}
\begin{proof}
	We build an auxiliary bipartite graph $\B$ with $1,\ldots,t$ as the vertices of the first
	part, and $\vph'_1,\ldots,\vph'_s$ as the vertices of the second part, and $i\vph'_j\in E(\B)$
	if $i\in L_j(v)$.  To extend the $\MM$-colorings $\vph_j'$, we seek a matching $M_{\B}$
	saturating the second part.  For this we use Hall's Theorem.

	Note that $d_{\B}(i)\ge s-d(v)$ and $d_{\B}(\vph'_j)\ge t-d(v)$, for all $i\in[t]$
	and $j\in[s]$.  Consider $S\subseteq\{\vph_1',\ldots,\vph_s'\}$.  Clearly 
	$|N_{\B}(S)|\ge t-d(v)\ge t/2$.  So if $|S|\le t/2$, then $|N_{\B}(S)|\ge |S|$
	as desired.  Otherwise, $|S|>t/2\ge d(v)$. So $|S|+d_{\B}(i) \ge |S|+s-d(v) > s$ because
	$|S|> d(v)$.  So by Pigeonhole $i\in N_{\B}(S)$ for all $i\in [t]$.  That is, $|N_{\B}(S)| = 
	t \ge |S|$, as desired.  Thus, by Hall's Theorem we have the desired matching $M_{\B}$ in $\B$.
\end{proof}

By definition, each matching in $\MM$ need not be perfect.  But to prove the results in this paper, it
suffices to consider the case that all matchings are perfect; if not, then we can grow them greedily. 
(In what follows, we implicitly assume that all matchings in $\MM$ are perfect.)

An \emph{$r$-sum} of graphs $G_1$ and $G_2$ is formed from the disjoint union $G_1+G_2$ by specifying 
an $r$-clique $x_1,\ldots,x_r$ in $G_1$ and an $r$-clique $y_1,\ldots,y_r$ in $G_2$, identifying $x_i$ with 
$y_i$ for each $i\in[r]$, and possibly deleting some edges of the resulting $r$-clique.  The \emph{Wagner 
graph} $M_8$ (also called the \emph{M\"{o}bius 8-ladder}) is formed from the $8$-cycle $z_1\cdots z_8$ by 
adding the 4 chords $z_iz_{i+4}$ for each $i\in [4]$.
Wagner~\cite{wagner} characterized the class of maximal $K_5$-minor-free graphs as follows.

\begin{thmA}[\cite{wagner}]
	\label{wagner-thm}
	If a graph $G$ is maximal $K_5$-minor-free, then $G$ can be built by a sequence of $2$-sums and 
	$3$-sums, without deleting edges, starting from maximal planar graphs and copies of $M_8$.  
\end{thmA}

	So to prove our Main Theorem, we would like to reduce to the case of planar graphs and copies of
	$M_8$.  However, we must ensure that our disjoint $\MM$-colorings $\vph_1,\vph_2,\vph_3$
	agree on the copies of $K_2$ and $K_3$ where we perform our 2-sums and 3-sums.  This motivates the
	following slightly stronger lemmas.  We handle the copies of $M_8$ first, even for 6 disjoint
	$\MM$-colorings, since this step is easy.

\begin{lem}
	\label[lem]{M8-lem}
	Given a copy of $M_8$ and a correspondence 6-cover $\MM$ of $M_8$, there exist disjoint
	$\MM$-colorings $\vph_1,\ldots,\vph_6$.  Furthermore, if $\vph_1^0,\ldots,\vph_6^0$ are
	disjoint $\MM$-colorings of a specified $K_2$ in $M_8$, then we can choose 
	$\vph_1,\ldots,\vph_6$ to extend $\vph_1^0,\ldots,\vph_6^0$.
\end{lem}
\begin{proof}
	We use induction on the number $n'$ of uncolored vertices.  The base case, $n'=0$, is trivial.
	And the induction step holds, because $M_8$ is 3-degenerate, by Lemma~\ref{ext-lem} with $s:=6$
	and $t:=6$.
\end{proof}

Now we consider planar graphs.

\begin{lem}
	\label[lem]{planar-lem}
	Let $G$ be a planar graph with a 6-correspondence cover $\MM$.  If $C$ is a $K_3$ in $G$,
	then each choice $\vph_1^0,\vph_2^0, \vph_3^0$ of disjoint $\MM$-colorings of $C$ extends to 
	disjoint $\MM$-colorings $\vph_1,\vph_2,\vph_3$ of $G$.
\end{lem}

The key to the proof of Lemma~\ref{planar-lem} is proving a still stronger statement, given below in the Key Lemma, that more easily facilitates 
proof by induction.  Our technique is known as a ``Thomassen-style'' proof, due to the strikingly short and 
elegant proof of Carsten Thomassen~\cite{Thomassen-5choosable} that all planar graphs are 5-choosable.  This 
method has since been applied many times~\cite{CCKPSZ, GZ, KT, LWZZ, thomassen-grotzsch2, WL, 
zhu-4choosable}.  For a unified study of numerous such examples, see~\cite[Chapter 11]{GCM}.

\begin{keylem}
	Let $G$ be a plane near-triangulation, with a correspondence 6-cover $\MM$.  Denote the vertices 
	on the outer face of $G$, in clockwise order, by $w_1,\ldots,w_n$.  There exist disjoint 
	$\MM$-colorings $\vph_1,\vph_2,\vph_3$ that extend $\vph_1^0,\vph_2^0,\vph_3^0$ and that have 
	$\vph_j(v)\in L_j(v)$ for all $j\in[3]$ if $\vph_1^0,\vph_2^0,\vph_3^0$ and $L_1,L_2,L_3$ 
	satisfy the following conditions
	\begin{enumerate}
		\item[(1)] $L_1(v)=L_2(v)=L_3(v)=[6]$ 
			for each vertex $v$ not on the outer face.
		\item[(2)] For each vertex $w_i$ with $i\in\{3,\ldots,n\}$ there exist distinct colors 
			$c_1,\ldots,c_6\in [6]$ with 
			$L_1(w_i)=\{c_1,c_2,c_3,c_4\}$ and
			$L_2(w_i)=\{c_1,c_2,c_5,c_6\}$ and
			$L_3(w_i)=\{c_3,c_4,c_5,c_6\}$.
		\item[(3)] 
			$\vph_1^0,\vph_2^0,\vph_3^0$ are disjoint $\MM$-colorings of $G[\{w_1,w_2\}]$.
	\end{enumerate}
\end{keylem}

\section{Proofs}

In this section we present proofs of our results.  We first prove part (i) of our Main Theorem, by using 
\Cref{M8-lem}, \Cref{planar-lem}, and Theorem~\ref{wagner-thm} (Wagner's characterization of 
$K_5$-minor-free graphs).  That is, given a $K_5$-minor-free graph 
$G$ and a correspondence 6-cover $\MM$, we show how to construct disjoint $\MM$-colorings 
$\vph_1,\vph_2,\vph_3$.  Next, we prove \Cref{planar-lem} using our Key Lemma; and after this we prove our 
Key Lemma via a ``Thomassen-style'' induction proof.  Finally, we conclude this section by proving part (ii) 
of our Main Theorem.  That is, we construct a $K_5$-minor-free graph $G$ and a correspondence 6-cover 
$\MM$ of $G$ that do not admit 3 disjoint $\MM$-colorings.

\begin{proof}[Proof of the Main Theorem.]
	Fix a $K_5$-minor-free graph $G$ and a correspondence 6-cover $\MM$.  It suffices to consider
	the case that $G$ is a maximal $K_5$-minor-free graph; if not, then we add edges to make it so.
	(The trivial case that $|V(G)|\le 2$ is handled easily by \Cref{ext-lem}.)

	By \Cref{wagner-thm}, $G$ can be constructed by a sequence of 2-sums and 3-sums starting from
	copies of $M_8$ and planar graphs.  Furthermore, by possibly reordering the sequence, we can assume
	that for each 2-sum or 3-sum one of the graphs being summed is itself either a copy of $M_8$ or a 
	planar graph.  Fix such a construction sequence, and let $q$ be the sum of its
	numbers of 2-sums and 3-sums.  We proceed by induction on $q$.  If $q=0$, then $G$ is planar or a 
	copy of $M_8$, so we are done by \Cref{M8-lem} or by \Cref{planar-lem}.  To be precise, in the latter
	case, we should first fix a copy $C$ of $K_3$ and disjoint $\MM$-colorings of $C$.  But this
	is easy by \Cref{ext-lem}.  This completes the base case, $q=0$.

	Now assume that $q\ge 1$.  So $G$ is formed from a 2-sum or 3-sum of graphs $G_1$ and $G_2$,
	where $G_2$ is planar or a copy of $M_8$ and $G_1$ is formed from a sequence of 2-sums and
	3-sums of length $q-1$.  By the induction hypothesis, $G_1$ has
	disjoint $\MM$-colorings $\vph_1',\vph_2',\vph_3'$.  Let the copy of $K_2$ or $K_3$ in $G_2$, 
	that is involved in the sum to form $G$, inherit from the corresponding vertices in $G_1$ three
	disjoint $\MM$-colorings $\vph_1^0,\vph_2^0,\vph_3^0$.  By \Cref{M8-lem} or \Cref{planar-lem}, 
	we can extend $\vph_1^0,\vph_2^0,\vph_3^0$ to disjoint $\MM$-colorings $\vph_1'',\vph_2'', 
	\vph_3''$ of $G_2$.  (If $G_2$ is planar and only a $K_2$ is precolored, then we first extend 
	the precoloring to a $K_3$, via \Cref{ext-lem}, before invoking \Cref{planar-lem}; this is possible 
	because $G_2$ is \emph{maximal planar}, so every $K_2$ is contained in a $K_3$.) Since 
	$\vph_j'$ and $\vph_j''$ agree on the vertices of $V(G_1)\cap V(G_2)$, for all $j\in[3]$, 
	together they give our disjoint $\MM$-colorings $\vph_1,\vph_2,\vph_3$ of $G$.
\end{proof}

Note that if we could prove a stronger version of \Cref{planar-lem} with 
4 disjoint $\MM$-colorings (or~5~or~6), then the above
proof  would give an an analogously stronger version of the Main Theorem.

\begin{proof}[Proof of \Cref{planar-lem}.]
	We now use the Key Lemma to prove \Cref{planar-lem}.  Fix a planar graph $G$ with a correspondence
	6-cover $\MM$, a copy $C$ of $K_3$ in $G$, and disjoint $\MM$-colorings $\vph_1^0,\vph_2^0,\vph_3^0$
	of $C$.  We use induction on $|V(G)|$; the base case $|V(G)|=3$ holds trivially.

	First suppose that $C$ is a separating cycle in $G$.  For each component $H_i$ of $G-V(C)$, let
	$G_i:=G[V(H_i)\cup V(C)]$.  By the induction hypothesis, the lemma holds for each $G_i$; that is,
	each $G_i$ has disjoint $\MM$-colorings $\vph_1^i,\vph_2^i,\vph_3^i$ that agree with 
	$\vph_1^0,\vph_2^0,\vph_3^0$ on $C$.  Since all these disjoint $\MM$-colorings agree on $C$,
	their union gives the desired disjoint $\MM$-colorings of $G$.

	So assume instead that $C$ is not a separating cycle in $G$; that is, $C$ is the boundary of a face.
	By redrawing if needed, we assume that $C$ is the boundary of the outer face of $G$; we denote its
	vertices by $w_1,w_2,w_3$.  Let $G':=G-w_3$.  For each $x\in N_G(w_3)$ and each $i\in [3]$, we delete 
	from $L_i(x)$ the color $\vph_i^0(w_3)$ as well as another arbitrary color so that $x$ will satisfy
	condition (2) of the Key Lemma.  
	Now conditions (1), (2), and (3) of the Key Lemma hold.  So we can
	extend these disjoint partial $\MM$-colorings to disjoint $\MM$-colorings 
	$\vph_1,\vph_2,\vph_3$ of $G$ that agree on $C$ with $\vph_1^0,\vph_2^0,\vph_3^0$.
\end{proof}

Now we prove our Key Lemma.
To denote the matching of $\MM$ for an edge $yz\in E(G)$, we write $\MM_{yz}$.  Furthermore, given a partial
$\MM$-coloring $\vph_i$ that colors a vertex $y$, we denote by $\MM_{yz}(\vph_i(y))$ the color for $z$
that is matched by $\MM_{yz}$ to the color $\vph_i(y)$ for $y$.  (And for an arbitrary subset $S$ of $[6]$
we let $\MM_{yz}(S):=\bigcup_{i\in S}\MM_{yz}(i)$.)

\begin{proof}[Proof of the Key Lemma.]
	Our proof is by induction on $|V(G)|$.  The base case, $|V(G)|=3$, is handled as follows.  We 
	construct a bipartite graph $\B$ with the vertices of one part being $\vph_1,\vph_2,\vph_3$,
	the vertices of the other part 
	being the colors $1,\ldots,6$ and $\vph_j$ being adjacent to a color $h$ if $h\in 
	L_j(w_3)\setminus\{\MM_{w_1w_3}(\vph^0_j(w_1)), \MM_{w_2w_3}(\vph_j^0(w_2))\}$.  
	That is, $\vph_j$ is \emph{not} adjacent to a color $h$ precisely when $h$ is either
	absent from $L_j$ or when the color is ``blocked'' from being used on $w_3$ by $\vph_j$, 
	due to the color used on $w_1$ or on $w_2$.  It suffices to find a matching in $\B$ that 
	saturates $\{\vph_1,\vph_2, \vph_3\}$; this will tell us how to extend the disjoint 
	$\MM$-colorings to $w_3$.  We will do this by Hall's Theorem.
	Note that $d_\B(\vph_j)\ge |L_j(v_3)|-2=2$ for all $j\in[3]$.  Furthermore, 
	$L_1(w_3)\cap L_2(w_3)\cap L_3(w_3)=\emptyset$.  Hence, 
	$|N_\B(\vph_1)\cup N_\B(\vph_2)\cup N_\B(\vph_3)|\ge 3$.  Thus, $\B$ satisfies the hypothesis of 
	Hall's Theorem, and we get the desired matching in $\B$ saturating $\{\vph_1,\vph_2,\vph_3\}$.

	Now we consider the induction step.  First suppose that $G$ contains a chord $w_iw_\ell$ of the cycle
	bounding the outer face.  By symmetry, we assume that $1\notin\{i,\ell\}$ (otherwise
	we relabel vertices in counterclockwise order, maintaining the fact 
	that $|L_j(w_h)|=1$ for each $h\in [2]$ and each $j\in[3]$).

	Let $C_1$ be the cycle induced by $w_1\cdots w_iw_\ell\cdots w_n$, and 
	$C_2$ be the cycle induced by $w_iw_{i+1}\cdots w_\ell$.  For each $h\in[2]$, let $G_h$ be the 
	subgraph of $G$ induced by the vertices on or inside of $C_h$.  By induction, we have disjoint
	$\MM$-colorings $\vph_1^1,\vph_2^1,\vph_3^2$ for $G_1$ that extend $\vph_1^0,\vph_2^0,\vph_3^0$.

	Now in $G_2$, we inherit from $G_1$ disjoint $\MM$-colorings of $G[\{w_i,w_{\ell}\}]$;
	denote them by $\vph_1^{00},\vph_2^{00},\vph_3^{00}$.
	Again, by induction, $G_2$ has disjoint $\MM$-colorings $\vph_1^2,\vph_2^2,\vph_3^2$, that
	extend $\vph_1^{00},\vph_2^{00},\vph_3^{00}$.
	Since $\vph_j^1$ and $\vph_j^2$ agree on $w_i,w_\ell$, for each $j\in[3]$, their unions give the 
	desired disjoint $\MM$-colorings $\vph_j$ for $G$.

	Now we assume instead that the boundary cycle $C$ of the outer face has no chord.
	Our plan is to delete $w_n$ and remove certain colors from the lists of neighbors of $w_n$, other
	than $w_1$ and $w_{n-1}$.  We get the desired disjoint $\MM$-colorings for this smaller graph by 
	induction, and afterward show that we can extend these colorings to the desired colorings for $G$.  
	Now we give more details.
	
	We will partition $[6]$ as $R_1\uplus R_2\uplus R_3$ such that, 
	for all $j\in [3]$, we have 
	\begin{align}
		\mbox{$|R_j|=2$ and $R_j\subseteq L_j(w_n)$ and $\MM_{w_1w_n}(\vph_j(w_1))\notin R_j$.}
		\label{R-condition}
	\end{align}
	(Here $R$ stands for ``reserved'' since these are colors in $L_j(w_n)$ that we do not allow to be 
	blocked by any neighbor of $w_n$, except for possibly $w_{n-1}$.)
	We first assume that we can find such a partition of $[6]$ satisfying \eqref{R-condition}, and 
	use this partition to construct the disjoint $\MM$-colorings.  Near the end of the proof, 
	we show how to find such a partition satisfying \eqref{R-condition}.
        Also, see \Cref{example1}.

	Let $G':=G-w_n$.  Let $L'_j(v):=L_j(v)\setminus \MM_{w_nv}(R_j)$ for each $j\in [3]$ and 
	each $v\in N_G(w_n) \setminus\{w_1,w_{n-1}\}$.  For all other $v\in V(G')$, let $L'_j(v):=L_j(v)$ 
	for all $j\in[3]$.  Now we must find disjoint $\MM$-colorings $\vph_1',\vph_2',\vph_3'$ for $G'$ 
	such that $\vph'_j(v)\in L'_j(v)$
	for all $v\in V(G')$ and all $j\in[3]$.  This essentially holds by the induction hypothesis, but
	we remark on a few details.  Because $G$ has no chords, each $v\in V(G')\setminus\{w_1,w_2\}$
	has $|L'_j(v)|\ge 4$ for all $j\in[3]$ and $\cup_{j=1}^3L'_j(v)=[6]$.  

	Let $\vph'_1,\vph'_2,\vph'_3$ denote the colorings of $G'$ guaranteed by the induction hypothesis.
	To extend $\vph'_j$ to $w_n$, we simply pick $\vph_j(w_n)$ from $R_j\setminus 
	\{\MM_{w_{n-1}w_n}(\vph_j'(w_{n-1}))\}$.
	This cannot create a conflict with any neighbor of $w_n$, since $\MM_{w_1w_n}(\vph_j'(w_1))\notin R_j$
	and by construction $\MM_{w_nx}(R_j)\cap L_j'(x) = \emptyset$ for each $x\in N_G(w_n)
	\setminus\{w_1,w_{n-1}\}$.

	\begin{center}
	\begin{tabular}{ccc|ccc}
		$\MM_{w_1w_n}(\vph_1(w_1))$ & $\MM_{w_1w_n}(\vph_2(w_1))$ & $\MM_{w_1w_n}(\vph_3(w_1))$ 
		& $R_1$ & $R_2$ & $R_3$\\
		\hline 
		\makecell{1\\1} & \makecell{2\\2} & \makecell{3\\5} & 23 & 15 & 46\\
		\hline
		\makecell{1\\1} & \makecell{5\\5} & \makecell{3\\6} & 23 & 16 & 45\\
		\hline
		\makecell{3\\3} & \makecell{1\\1} & \makecell{4\\5} & 14 & 25 & 36\\
		\hline
		\makecell{3\\3} & \makecell{5\\5} & \makecell{4\\6} & 14 & 26 & 35\\
		\hline
	\end{tabular}
		\captionsetup{width=.775\textwidth}
		\captionof{table}{Left: The 8 possibilities for 
		$(\MM_{w_1w_n}(\vph_1(w_1)),\MM_{w_1w_n}(\vph_2(w_1)),\MM_{w_1w_n}(\vph_3(w_1)))$,
		up to permuting colors. Right: Our 4 partitions of $[6]$ as $R_1\uplus R_2\uplus R_3$.%
		\label{table1}}
	\end{center}

	Finally, we construct our partition of $[6]$ as
	$R_1\uplus R_2\uplus R_3$.  By Condition (2) of the hypothesis, and by possibly 
	renaming colors, we assume that $L_1(w_n)=\{1,2,3,4\}$, $L_2(w_n)=\{1,2,5,6\}$, and 
	$L_3(w_n)=\{3,4,5,6\}$.  We assume that $\MM_{w_1w_n}(\vph_j(w_1))\in L_j(w_n)$ for 
	each $j\in [3]$; doing so makes our task no easier.  
	We can swap the names of colors 1 and 2, and 3 and 4, and 5 and 6.
	So, by symmetry, we need only consider the 8 possibilities for 
	$(\MM_{w_1w_n}(\vph_1(w_1)),\MM_{w_1w_n}(\vph_2(w_1)),\MM_{w_1w_n}(\vph_3(w_1)))$ 
	shown in Table~\ref{table1}.  In each case we use the
	corresponding partition shown on the right.  It is straightforward to check that 
	this partition has the desired properties.  This finishes the proof.
\end{proof}

We used symmetry twice to simplify the proof of the Key Lemma.  This makes the proof more compact, but
harder to implement.  For concreteness, we work through the example
in Figure~\ref{example1-fig}.  

\newlength{\outerradius}
\setlength{\outerradius}{3.4pt}

\newcommand{\partshadedcircle}[4]{ 
  \begin{scope}[shift=#3, rotate=#4]
    \fill[gray!25!blue!65] (0,0) -- (90-#1:\outerradius)
      arc (90-#1:90-#2:\outerradius) -- (0,0);
  \draw[thin] circle (\outerradius+.25pt);
  \end{scope}
}

\newcommand\mvertA{\partshadedcircle{0}{120}}
\newcommand\mvertB{\partshadedcircle{120}{240}}
\newcommand\mvertC{\partshadedcircle{240}{360}}
\newcommand\mvertD{\partshadedcircle{0}{240}}
\newcommand\mvertE{\partshadedcircle{-120}{120}}
\newcommand\mvertF{\partshadedcircle{120}{360}}
\newcommand\mvertG{\partshadedcircle{0}{360}}
\newcommand\mvertH{\partshadedcircle{0}{0}}

\begin{figure}[!b]
	\begin{tikzpicture}[semithick, scale=.85] 
		\tikzstyle{lStyle}=[shape = circle, minimum size = .85*8pt, inner sep = 1.0pt, outer sep = 1pt, draw, fill=white]
	\tikzset{every node/.style=lStyle}
	\def\dist{.35cm}
	\def\myscale{2}

	\newcommand\mynode[4]{  
		\begin{scope}[xshift=\myscale*#1 cm, yshift=\myscale*#2 cm, rotate=#3]
		  \draw[fill=white] (-2.5*\dist,0) ++ (-.7*\dist,-.6*\dist) --++ (6.4*\dist,-.3*\dist) --++ 
			(0,1.8*\dist) --++ (-6.4*\dist,-.3*\dist) -- cycle;	
		  \draw (-2.5*\dist,0) node (v1#4) {} ++ (\dist,0) node (v2#4) {} ++ (\dist,0) node (v3#4) {}
		  	  ++ (\dist,0) node (v4#4) {} ++ (\dist,0) node (v5#4) {} ++ (\dist,0) node (v6#4) {};
		\end{scope}
	}
	
		\draw[gray, line width=.1in] (3.80,-4.6) -- (5.00,-5.05); 
		\draw[gray, line width=.1in] (1.225,-2.5) -- (2.1,-3.375); 
		\draw[gray, line width=.107in] (0.59,0.6) -- (0.59,-0.6); 
		\draw[gray, line width=.1in] (1.225,2.5) -- (2.1,3.375); 
		\draw[gray, line width=.1in] (3.80,4.6) -- (5.00,5.05); 

        \draw (7.45,-5.05) node[draw=none] {\footnotesize {$w_1$}};
        \draw (7.65,4.95) node[draw=none] {\footnotesize {$w_{n-1}$}};
        \draw (9.05,-1.35) node[draw=none] {\footnotesize {$w_{n}$}};
        
		\mynode{4.5}{0}{-90}{a}
		\mynode{4.5}{0}{-90}{a}
		\mynode{4.5}{0}{-90}{a}
		\mvertD{(v1a)}{0}
		\mvertE{(v2a)}{0}
		\mvertF{(v3a)}{0}
		\mvertD{(v4a)}{0}
		\mvertF{(v5a)}{0}
		\mvertE{(v6a)}{0}

		\mynode{3}{-2.5}{0}{b}
		\mvertA{(v1b)}{0}
		\mvertH{(v2b)}{0}
		\mvertB{(v3b)}{0}
		\mvertH{(v4b)}{0}
		\mvertC{(v5b)}{0}
		\mvertH{(v6b)}{0}

		\mynode{1.5}{-2}{-35}{c}
		\mvertG{(v1c)}{0}
		\mvertG{(v2c)}{0}
		\mvertG{(v3c)}{0}
		\mvertG{(v4c)}{0}
		\mvertG{(v5c)}{0}
		\mvertG{(v6c)}{0}

		\mynode{.485}{-.8}{-69}{d}
		\mvertG{(v1d)}{0}
		\mvertG{(v2d)}{0}
		\mvertG{(v3d)}{0}
		\mvertG{(v4d)}{0}
		\mvertG{(v5d)}{0}
		\mvertG{(v6d)}{0}

		\mynode{.485}{.8}{-111}{e}
		\mvertG{(v1e)}{0}
		\mvertG{(v2e)}{0}
		\mvertG{(v3e)}{0}
		\mvertG{(v4e)}{0}
		\mvertG{(v5e)}{0}
		\mvertG{(v6e)}{0}
		
		\mynode{1.5}{2}{35-180}{f}
		\mvertG{(v1f)}{0}
		\mvertG{(v2f)}{0}
		\mvertG{(v3f)}{0}
		\mvertG{(v4f)}{0}
		\mvertG{(v5f)}{0}
		\mvertG{(v6f)}{0}

		\mynode{3}{2.5}{180}{g}
		\mvertE{(v1g)}{0}
		\mvertF{(v2g)}{0}
		\mvertE{(v3g)}{0}
		\mvertD{(v4g)}{0}
		\mvertD{(v5g)}{0}
		\mvertF{(v6g)}{0};

		\draw [thin, color=gray]
		(v1a) -- (v3b) 
		(v4a) -- (v1b)
		(v6a) -- (v5b)

		(v1a) -- (v3c) 
		(v2a) -- (v5c)
		(v3a) -- (v2c)
		(v4a) -- (v6c)
		(v5a) -- (v1c)
		(v6a) -- (v4c)

		(v1a) -- (v3d) 
		(v2a) -- (v2d)
		(v3a) -- (v6d)
		(v4a) -- (v5d)
		(v5a) -- (v4d)
		(v6a) -- (v1d)

		(v1a) -- (v1e) 
		(v2a) -- (v5e)
		(v3a) -- (v2e)
		(v4a) -- (v3e)
		(v5a) -- (v6e)
		(v6a) -- (v4e)

		(v1a) -- (v4f) 
		(v2a) -- (v3f)
		(v3a) -- (v1f)
		(v4a) -- (v6f)
		(v5a) -- (v2f)
		(v6a) -- (v5f)

		(v1a) -- (v5g) 
		(v2a) -- (v6g)
		(v3a) -- (v3g)
		(v4a) -- (v4g)
		(v5a) -- (v1g)
		(v6a) -- (v2g);

		\begin{scope}[xshift=4.5in]

		\draw[gray, line width=.1in] (3.80,-4.6) -- (5.00,-5.05); 
		\draw[gray, line width=.1in] (1.225,-2.5) -- (2.1,-3.375); 
		\draw[gray, line width=.107in] (0.59,0.6) -- (0.59,-0.6); 
		\draw[gray, line width=.1in] (1.225,2.5) -- (2.1,3.375); 
		\draw[gray, line width=.1in] (3.80,4.6) -- (5.00,5.05); 

        \draw (7.45,-5.05) node[draw=none] {\footnotesize {$w_1$}};
        \draw (7.65,4.95) node[draw=none] {\footnotesize {$w_{n-1}$}};

		\mynode{3}{-2.5}{0}{b}
		\mvertA{(v1b)}{0}
		\mvertH{(v2b)}{0}
		\mvertB{(v3b)}{0}
		\mvertH{(v4b)}{0}
		\mvertC{(v5b)}{0}
		\mvertH{(v6b)}{0}

		\mynode{1.5}{-2}{-35}{c}
		\mvertE{(v1c)}{0}
		\mvertD{(v2c)}{0}
		\mvertF{(v3c)}{0}
		\mvertF{(v4c)}{0}
		\mvertD{(v5c)}{0}
		\mvertE{(v6c)}{0}

		\mynode{.485}{-.8}{-69}{d}
		\mvertF{(v1d)}{0}
		\mvertD{(v2d)}{0}
		\mvertF{(v3d)}{0}
		\mvertE{(v4d)}{0}
		\mvertE{(v5d)}{0}
		\mvertD{(v6d)}{0}

		\mynode{.485}{.8}{-111}{e}
		\mvertF{(v1e)}{0}
		\mvertD{(v2e)}{0}
		\mvertE{(v3e)}{0}
		\mvertF{(v4e)}{0}
		\mvertD{(v5e)}{0}
		\mvertE{(v6e)}{0}
		
		\mynode{1.5}{2}{35-180}{f}
		\mvertD{(v1f)}{0}
		\mvertE{(v2f)}{0}
		\mvertD{(v3f)}{0}
		\mvertF{(v4f)}{0}
		\mvertF{(v5f)}{0}
		\mvertE{(v6f)}{0}

		\mynode{3}{2.5}{180}{g}
		\mvertE{(v1g)}{0}
		\mvertF{(v2g)}{0}
		\mvertE{(v3g)}{0}
		\mvertD{(v4g)}{0}
		\mvertD{(v5g)}{0}
		\mvertF{(v6g)}{0}

		\end{scope}
	
\end{tikzpicture}
\caption{Left: The portion of $\MM$ induced by $N[w_n]$. Right: The portion induced by $N[w_n]-\{w_n\}$. (Each thick gray line, say between vertices $x$ and $y$, denotes a perfect matching between $[6]_x$ and 
$[6]_y$. All other aspects of the figure are explained in \Cref{example1}.)%
\label{example1-fig}}
\end{figure}
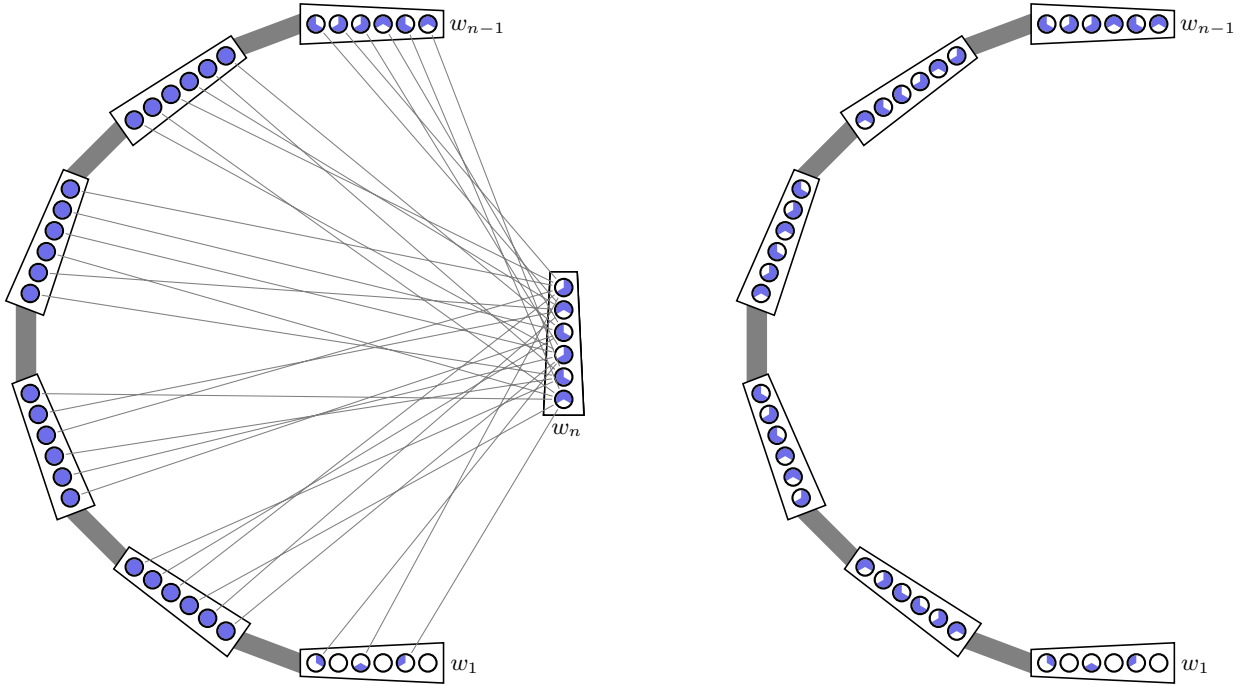

\begin{example}
\label{example1}
Figure~\ref{example1-fig} only shows  $N[w_n]$.  The rightmost 3 vertices on the left side are 
$w_{n-1}$, $w_n$, and $w_1$.  Each trapezoid encloses the numbers $1,\ldots,6$ for some vertex (with
1 at the narrow end and 6 at the wide end).  The shading of a vertex shows the $\MM$-colorings $\vph_j$
for which the corresponding color is available.  The first third of the circle (12:00-4:00) represents
$\vph_1$, and the second and final thirds represent, respectively, $\vph_2$ and $\vph_3$.  
So the figure shows that $\vph_1(w_1)=1$, $\vph_2(w_1)=3$, $\vph_3(w_1)=5$.  (For clarity,
we omit the matching edges incident to other colors at $w_1$, since they are irrelevant.)

From the figure, we see that $L_1(w_n)=\{1,2,4,6\}$, $L_2(w_n)=\{1,3,4,5\}$, and $L_3(w_n)=\{2,3,5,6\}$.
For brevity, henceforth we shorten $L_j(w_n)$ to $L_j$.  In the proof of the Key Lemma, we assume by symmetry 
that these lists are $\{1,2,3,4\}$, $\{1,2,5,6\}$, and $\{3,4,5,6\}$.  So we need a permutation $\pi$ to 
justify this assumption.  Let $\pi:=(1)(2364)(5)$ and let $\overline{L_j}:=\{\pi(h):h\in L_j\}$.  Note that 
$\overline{L_1}=\{1,2,3,4\}$, $\overline{L_2}=\{1,2,5,6\}$, and $\overline{L_3}=\{3,4,5,6\}$, as desired. 

The vector of colors forbidden by $w_1$ from use on $w_n$ for colorings $\vph_1,\vph_2,\vph_3$ is $(4,1,6)$.
But we are actually interested in $(\pi(4),\pi(1),\pi(6))=(2,1,4)$.  We look for a corresponding line in
\Cref{table1}, but do not immediately see one.  So we need a permutation that swaps one or more of the pairs
$1,2$ and $3,4$ and $5,6$.  Let $\rho:=(12)(34)(5)(6)$.  Now $(\rho(2),\rho(1),\rho(4))=(1,2,3)$, which is
the first line of the table.  (Note that if we let $\overline{\overline{L_j}}:=\{\rho(h):h\in \overline{L_j}\}$,
then $\overline{\overline{L_j}}=\overline{L_j}$ for all $j\in [3]$.) Based on the table, we let 
$\overline{\overline{R_1}} :=\{2,3\}$, $\overline{\overline{R_2}}:=\{1,5\}$, 
$\overline{\overline{R_3}}:=\{4,6\}$.  Now let $\overline{R_j}:=\{\rho^{-1}(h):h\in 
\overline{\overline{R_j}}\}$.  So $\overline{R_1}=\{1,4\}$, $\overline{R_2}=\{2,5\}$, 
$\overline{R_3}=\{3,6\}$.  Finally, let $R_j:=\{\pi^{-1}(h):h\in\overline{R_j}\}$.  
So $R_1=\{1,6\}$, $R_2=\{4,5\}$, $R_3=\{2,3\}$.  Finally, we follow the matching edges from $R_1$, $R_2$, 
$R_3$ to see what colors are excluded from each $L_j$ at each neighbor of $w_n$, when by the induction 
hypothesis we find the desired disjoint $\MM$-colorings for $G-w_n$.
\end{example}

To finish this section, we prove part (ii) of our Main Theorem.

\begin{lem}\label{lem:construction_K3t}
	\label{part-ii-lem}
	Let $G:=K_{3,120^3}$.  There exists a correspondence 5-cover $\MM$ for $G$ such that $G$ does not
	have 3 disjoint $\MM$-colorings.  This proves part (ii) of the Main Theorem, 
	since $K_{3,120^3}\in \K_5$.
	
\end{lem}
\begin{proof}
	Denote the parts of $G$ by $U$ and $W$, where $U=\{u_1,\ldots,u_{120^3}\}$ and $W=\{x,y,z\}$.
	We first prove the second statement.  If $G$ contains a $K_5$-minor, then $V(G)$ contains
	disjoint subsets $V_1,\ldots,V_5$ that when contracted give $V(K_5)$.
	These sets are disjoint, so at least two contain no vertex of $W$.  Thus their corresponding
	vertices in the contraction are non-adjacent, a contradiction.

	Now we prove the first statement.
	For each $v\in V(G)$, denote by $[5]_v$ the set $\{1_v,\ldots,5_v\}$ of allowable colors for $v$.
	Let $\mathcal{M}$ denote the set of all $5!$ possible perfect matchings between $[5]$ and $[5]$.
	We now build our 5-cover $\MM$ of $G$.
	For each triple $(M_x,M_y,M_z) \in \mathcal{M}\times\mathcal{M}\times\mathcal{M}$, choose a unique 
	vertex $u_i$ in $U$, and include matching $M_x$ between $[5]_x$ and $[5]_{u_i}$, 
	matching $M_y$ between $[5]_y$ and $[5]_{u_i}$, and matching $M_z$ between $[5]_z$ and $[5]_{u_i}$.

	Suppose, for a contradiction, that $G$ admits 3 disjoint $\MM$-colourings 
	$\vph_1, \vph_2, \vph_3$.  Here each $\vph_i$ is a function from $V(G)$ to [5].
	Let
	$N_x := \{(\vph_1(x),1)$, $(\vph_2(x),2),(\vph_3(x),3)\}$ and
	$N_y := \{(\vph_1(y),3)$, $(\vph_2(y),1),(\vph_3(y),2)\}$ and
	$N_z := \{(\vph_1(z),2)$, $(\vph_2(z),3),(\vph_3(z),1)\}$.
	Choose a vertex $u_i$ such that the matchings between 
	$x$ and $u_i$, $y$ and $u_i$, and $z$ and $u_i$ are $N_x, N_y, N_z$.

Due to the three cover-edges in the first coordinate, we must have $\vph_1(u_i) \in \{4,5\}$.
Due to the second coordinate, $\vph_2(u_i) \in \{4,5\}$; and due to the third coordinate, $\vph_3(u_i) \in \{4,5\}$.  Therefore, $\vph_1(u_i),\vph_2(u_i),\vph_3(u_i)$ cannot be disjoint, a contradiction. 
\end{proof}

We observe that \Cref{part-ii-lem} is sharp as follows.  For every correspondence 5-cover $\MM$, and every
$\MM$-coloring $\vph_1$, there exists a disjoint $\MM$-coloring $\vph_2$.  This holds for every graph
$K_{3,r}$ by induction on $r$, using an argument similar to (but simpler than) that proving \Cref{ext-lem}.
In fact, the same is true of every graph that is 3-degenerate.
Furthermore, Lemma~\ref{lem:construction_K3t} should also hold for $K_{3,t}$ with $t$ much smaller than $120^3$, and the ideas in~\cite{CH23} should help to achieve such an optimization. 
\bigskip

Motivated by Hadwiger's Conjecture, we propose proving better bounds on $\chisc(\K_s)$ for all $s\ge 6$.
Determining such values precisely for general $s$ seems hard.  (In particular, if they are linear in $s$,
then a proof would imply the famous Linear Hadwiger Conjecture: There exists a constant $a$ such that if
$G$ is $K_s$-minor-free, then $\chi(G)\le as$.) But perhaps the values can be determined, or bounded tightly,
for some further small values of $s$.

In this direction, we mention that $\chisc(\K_s)>2s-5$ for all integers $s$.  Specifically, this is 
witnessed by the complete bipartite graph $\K_{s-2,r}$ when we let $r:=((2s-5)!)^{s-2}$.  To construct our 
$(2s-5)$-cover $\MM$, we generalize the approach in the proof of \Cref{part-ii-lem}, taking all possible 
ordered sets of matchings from the small side to the big side.  It was previously 
observed~\cite[Proposition~24]{CCvBDK} that this cover does not admit $2s-5$ disjoint $\MM$-colorings. 
In fact, the same argument in the previous proof shows that
it does not even contain $s-2$ disjoint $\MM$-colorings.
\bigskip

We end with a question.
\begin{ques}\label{question_twodisjoint}
Let $G$ be a $K_5$-minor-free graph. Does every $5$-cover $\MM$ of $G$
admit two disjoint $\MM$-colorings?
\end{ques}
If the answer to Question~\ref{question_twodisjoint} is yes, then it directly implies the first part of our Main Theorem, as from a $6$-cover $\MM$ one can choose an arbitrary $\MM$-coloring $\vph$, construct a $5$-cover $\MM_2$ from $\MM$ by removing $\vph(v)$ for every vertex $v$, and then find two disjoint $\MM_2$-colorings. On the other hand, if the answer is no, then it will underscore that $6$-covers are essential to obtain three disjoint colorings.

\footnotesize{
\bibliographystyle{habbrv}
\bibliography{thomassen}
}

\end{document}